\documentclass[preprint,12pt]{elsarticle}

\usepackage{amssymb}
\usepackage{amsthm}
\usepackage{amsmath}
\usepackage{mathtools}
\newtheorem{definition}{Definition}[section]
\newtheorem{thm}{Theorem}[section]
\newtheorem{corollary}{Corollary}[thm]
\newtheorem{lemma}{Lemma}[section]

\newtheorem{prop}{Proposition}[section]
\usepackage{amsfonts}
\usepackage{dsfont}

\newcommand{\menor}{<}

\usepackage{amssymb}

\usepackage{url}
\usepackage{amsthm}
\usepackage{multicol}

\journal{Journal of Stochastic Analysis and Applications}

\begin{document}

\begin{frontmatter}



\title{An Itô-Wentzell formula for the fractional Brownian motion}


\author[1,2]{Luís da Maia \\ \url{luis.maia@uni.lu}}

\affiliation[1]{organization={GFM and Dep. de Matemática, Instituto Superior Técnico}, city={Lisboa}, country={Portugal}}

\affiliation[2]{organization={Department of Mathematics, Luxembourg University},  country={Luxembourg}}

%\author[inst2]{Luís da Maia \\ \url{luis.maia@uni.lu}}

\begin{abstract}
 We prove an Itô-Wentzell formula for the fractional Brownian motion. As an application we derive an existence and uniqueness result for a class of stochastic differential equations driven by this stochastic process.
%The fractional Brownian motion is a generalization of the usual Brownian motion with several applications in mathematical modeling across various fields. Our aim is to study a generalization of Itô's formula when the function itself is random. Thus, we are going to construct the stochastic integral with respect to this process and we'll prove the Itô-Wentzell formula. With this result we'll prove existence and uniqueness of solution for a class of stochastic differential equations driven by fractional Brownian motion.
\end{abstract}



\begin{keyword}
Fractional Brownian motion \sep Itô-Wentzell Formula
%%\PACS 0000 \sep 1111
\MSC 60G22 \sep 60H10
\end{keyword}

\end{frontmatter}


\section{Introduction}

Itô-Wentzell formulae are extensions of the Itô chain rule where not only deterministic functions but also random processes can be considered in the composition. They were initially proved by Wentzell in \cite{wentzell-1965} and later on generalized by several authors, in different contexts.

Such formulae have proved to be very useful in many applications, for the study of stochastic partial differential equations, in filtering or in mathematical finance, for instance.

In this work we prove an Itô-Wentzell formula in the context of stochastic calculus for fractional Brownian motions. We use stochastic integrals as defined by Duncan and Hu in \cite{duncan-2000}.
As an application we study a particular class of stochastic differential equations driven by fractional Brownian motion.

In \cite{castrquini-2022} the authors also studied an Itô-Wentzell formula to Young integral that allows computation of the composition of $\alpha$-Hölder paths with $\alpha \in \left(\frac{1}{2},1 \right]$.

The plan of the exposition goes as follows: In section 2, we present the theory of integration with respect to the fractional Brownian motion using Wick product as in \cite{duncan-2000}. In section 3, we derive a stochastic rule for the product of two processes driven by fractional Brownian motion and then present and prove the main result of this paper, the Itô-Wentzell type formula for fractional Brownian motion. Finally, in section 4, we use the Itô-Wentzell formula previously proved to study the existence and uniqueness of solution for a class of stochastic differential equations driven by fractional Brownian motion. 

\section{Integration with respect to fBm when $H > \frac{1}{2}$}
\label{sec:int_fBm}
Fractional Brownian motion (fBm) is a stochastic process which generalizes the standard Brownian motion. Its properties depend on a parameter $H$, called the Hurst parameter. 

Let $\Omega = C_0(\mathbb{R}_+, \mathbb{R})$ be the space of continuous functions on $\mathbb{R}_+$ with initial value zero and the topology of local uniform convergence. There is a probability measure $\mathbb{P}^H$ on $(\Omega, \mathcal{F})$, where $\mathcal{F}$ is the Borel $\sigma$-algebra such that on the probability space $(\Omega, \mathcal{F}, \mathbb{P}^H)$ the coordinate process $W^H:\Omega \rightarrow \mathbb{R}$ given by

\begin{equation*}
    W_t^H(\omega)=\omega(t) \quad \quad \omega \in \Omega.
\end{equation*}

\vskip 5mm

\noindent satisfies the following definition (see \cite{duncan-2000}):

\begin{definition}\label{def:1}[Fractional Brownian Motion (fBm)] Let $H \in (0,1)$ be a constant and $W^H_t = W^H(t,\omega): [0,\infty)\times \Omega \rightarrow \mathbb{R}$ be a centered Gaussian process, such that
    \begin{equation*}
        \begin{cases}
        \mathbb{E}[W_t^H] = 0, \\
        R_H(s,t) := \mathbb{E}[W_s^H W_t^H] = \frac{1}{2}\{t^{2H} + s^{2H} - |t-s|^{2H} \},
        \end{cases}
    \end{equation*}
where $R_H(s,t)$ denotes the covariance function for the fBm with parameter $H$.

We call $W_t^H$ a Fractional Brownian Motion (fBm) with Hurst parameter $H$.
\end{definition}

We notice that when $H=\frac{1}{2}$ the covariance function becomes $\frac{1}{2}\{t^{2H} + s^{2H} - |t-s|^{2H} \} =\frac{1}{2}\{t + s - |t-s|\} = \text{min}(s,t)$, which is the exact covariance function for the standard Brownian motion.

\nocite{nourdin-2012}

%Two aspects of fBm are quite problematic: the process is not Markov nor a martingale when $H\neq \frac{1}{2}$ (\cite{duncan-2000}, \cite{biagini-2008}). Therefore, the construction of the stochastic integral becomes a more difficult task. 

The stochastic integration that we present works for $H>\frac{1}{2}$ and relies on a function $\phi$ presented below.

%%\section{Integration with respect to fBm when $H > \frac{1}{2}$}

%alteração, confirmar
Let $\phi:\mathbb{R}_+ \times \mathbb{R}_+ \rightarrow \mathbb{R}_+$ be such that $\phi(s,t) = H(2H-1)|s-t|^{2H-2}$.
Then we can write the covariance function of the fBm (for $H > \frac{1}{2}$) as:

\begin{equation*}
\mathbb{E}[W^H_s W^H_t] = R(s,t) = \int_0^t \int_0^s \phi(u,v)dudv.
\end{equation*}

Let's consider $\mathcal{S}(\mathbb{R})$, the Schwartz space of rapidly decreasing functions.
and the inner product in $\mathcal{S}(\mathbb{R})$ defined by,

\begin{equation*}
    \langle f,g \rangle_{\phi}= \int_\mathbb{R}\int_\mathbb{R} f(t)g(s)\phi(s,t)dsdt;
\end{equation*}
then the completion of $\mathcal{S}(\mathbb{R})$, with respect to the norm induced by this inner product is denoted by $L^2_\phi(\mathbb{R})$.

The stochastic integral of deterministic functions with respect to fBm is, of course, a random variable. It can be obtained by approximating the integral of simple functions by a Riemann sum with the fBm and then proceeding with a density argument.

\begin{prop}
If $f$ and $g$ are deterministic functions in $ L^2_\phi (\mathbb{R})$, then the processes $\int_{\mathbb{R}} f(s) dW_s^H$ and $\int_{\mathbb{R}} g(s) dW_s^H$ are well defined Gaussian processes with mean 0 and variances $||f||_{\phi}^2$ and $||g||_{\phi}^2$, resp, and 
\begin{equation*}
    \mathbb{E}\left[ \int_{\mathbb{R}} f(s) dW_s^H \int_\mathbb{R} g(s) dW_s^H\right] = \int_\mathbb{R} \int_\mathbb{R} f(s)g(t)\phi(s,t) dsdt = \langle f,g \rangle_\phi
\end{equation*}

\end{prop}

\noindent Define the space

\begin{definition}
    The space $L^p(\Omega) := L^p$, for $p\geq1$, is the space of all random variables $F:\Omega \rightarrow \mathbb{R}$, such that:
    \begin{equation*}
        ||F||_{\Omega} = \mathbb{E}[|F|^p]^{\frac{1}{p}} < \infty
    \end{equation*}
\end{definition}

\begin{definition}[Exponential Functional]
For any deterministic function, $f \in L^2_\phi(\mathbb{R})$, we define the exponential functional as,

\footnotesize
\begin{equation*}    
    \varepsilon(f) = \exp \left\{\left(\int_{\mathbb{R}} f(s) dW_s^H - \frac{1}{2}\int_{\mathbb{R}} f(s)f(t)\phi(s,t) dsdt \right) \right\} = \exp \left\{ \left(\int_{\mathbb{R}} f(s) dW_s^H - \frac{1}{2}||f||^2_{\phi} \right) \right\}
\end{equation*}
\end{definition}

\normalsize

We note that if $f \in L^2_\phi(\mathbb{R})$ then $\varepsilon(f) \in L^p$. The following properties hold:

\begin{thm}
    The linear span of the exponential functionals, 
     
    $\mathcal{E} = \left\{ \sum_{k=1}^n a_k \varepsilon(f_k), n \in \mathbb{N}, a_k \in \mathbb{R}, f_k \in L^2_\phi(\mathbb{R}) \right\}$, is dense in $L^p$ for $p \geq 1$. 
\end{thm}

\begin{proof}
    See \cite{duncan-2000}, Theorem 3.1.
\end{proof}

\begin{thm}
If $f_1, f_2, ..., f_n$ are elements of $L_\phi^2$ such that $||f_i-f_j||_\phi \neq 0$ for $i \neq j$, then $\varepsilon(f_1), \varepsilon(f_2), ..., \varepsilon(f_n)$ are linearly independent in $L^2_\phi$.

\end{thm}

\begin{proof}
    See \cite{duncan-2000}, Theorem 3.2.
\end{proof}

\begin{thm}
    If $F:\Omega \rightarrow \mathbb{R}$ is a random variable, such that $F \in L^p(\mathbb{P}^H)$, $p \geq 1$ then
    \footnotesize
    \begin{equation*}
        \mathbb{E} \left\{F\left(W^H + \int_0^. (\Phi g)(s)ds\right)  \right\} = \mathbb{E} \left\{F(W^H) \exp{\left(\int_{\mathbb{R}} g(s) dW_s^H - \frac{1}{2}\int_{\mathbb{R}} g(s)g(t)\phi(s,t) dsdt \right)}  \right\}  
    \end{equation*}
    \normalsize
    where $(\Phi g)(t) = \int_0^\infty \phi(t,u)g(u)du$  and $g \in L^2_\phi$.
\end{thm}

\begin{proof}
    See \cite{duncan-2000}, Theorem 3.3.
\end{proof}

Now we introduce an analogue of the Malliavin derivative.

\begin{definition}
    The $\phi$-derivative of a random variable $F \in L^p$ in the direction of $\Phi g$ is defined as
    \begin{equation*}
        D_{\Phi g}F(\omega) = \lim_{\delta \rightarrow 0} \frac{1}{\delta} \left\{F\left(\omega + \delta \int_0^. (\Phi g)(u) du \right)- F(\omega) \right\} \quad \quad a.s. \quad \omega \in \Omega
    \end{equation*}
    given that this limit exists.
    If there exists a process $(D^{\phi} F_s, s\geq 0)$ such that
    \begin{equation*}
        D_{\Phi g} F = \int_0^\infty D^{\phi}_s F g_s ds
    \end{equation*}
    for all $g \in L^2_{\phi}$ then $F$ is said to be $\phi$-differentiable.
\end{definition}

Accordingly, if instead of having a random variable, we have a stochastic process, we define it as being $\phi$-differentiable:

\begin{definition}
    Let $F_t: [0,T] \times \Omega \rightarrow \mathbb{R}$ be a stochastic process. This process is $\phi-$differentiable if for each $t$, $F(t,.)$ is  $\phi-$differentiable and $(D^\phi F_t)_s$, denoted by, $D_s^{\phi}F_t$, is jointly measurable.
\end{definition}

To extend the stochastic calculus from Brownian motion to the fBm, we need to introduce the notion of Wick product. 

\begin{definition}[Wick Product of  exponential functionals]
The Wick product, denoted by $\diamond$ for two exponential functionals, $\varepsilon(f)$ and $\varepsilon(g)$, is defined as $\varepsilon(f) \diamond \varepsilon(g) = \varepsilon(f+g)$

\end{definition}

We notice that for distinct $f_1$, $f_2$,..., $f_n$ $\in L_{\phi}^2$, their exponential functions are linearly independent. Futhermore $\mathcal{E}$, the set of linear combinations of exponentials, is dense in $L^p$, so we deduce that the definition of Wick product of two general functionals can be given by a density argument.

\begin{prop}\label{prop:3}
    If $g \in L_\phi^2$, $F \in L^2$ and $D_{\Phi g} \in L^2$, then 
    \begin{equation*}\label{eq:1}
        F \diamond \int_0^\infty g_s dW_s^H = F \int_0^\infty g_s dW_s^H - D_{\Phi g}F
    \end{equation*}
\end{prop}

\begin{proof}
    See \cite{duncan-2000}, Proposition 3.4.
\end{proof}

\begin{thm}
    If $g \in L_\phi^2$ and $\mathcal{E}_g$ is the completion of $\mathcal{E}$ under the norm 
    \begin{equation*}
        ||F||_g^2 = \mathbb{E}\left[(D_{\Phi g}F)^2 + F^2 |g|_\phi^2 \right],
    \end{equation*}
    then for any $F \in \mathcal{E}_g$, the process $F \diamond \int_0^\infty g_s dW_s^H$ is well defined and 
    \begin{equation*}
        \mathbb{E}\left[F\diamond \int_0^\infty g_s dW_s^H \right]^2 = \mathbb{E}\left[(D_{\Phi g}F)^2 + F^2 |g|_\phi^2 \right]
    \end{equation*}
\end{thm}

\begin{proof}
    See \cite{duncan-2000}.
\end{proof}

\begin{corollary}\label{corr:1}
Let $g,h\in L^2_\phi$ and $F,G \in \mathcal{E}$. Then,
\begin{equation*}
    \mathbb{E}\left[\left(F \diamond \int_0^\infty g_s dW_t^H \right) \left(G \diamond \int_0^\infty h_s dW_s^H \right) \right] = \mathbb{E}[D_{\Phi g}F D_{\Phi h}G + FG\langle g,h \rangle_\phi].
\end{equation*}
\end{corollary}

We are now ready to  study integration with respect to the fractional Brownian motion.

First we take a process $F$ in $\mathcal{E}$.
Let us consider an arbitrary partition of $[0,T]$, $\pi: 0 = t_0< t_1 < ... < t_n = T$. Denote by $|\pi|:=\max_{i}(t_{i+1}-t_i)$ and $F_t^\pi = F_{t_i}$ if $t_i \leq t \menor t_{i+1}$.

With this partition we consider the Riemann sum given by the Wick product,

\begin{equation*}
    S(F,\pi) = \sum_{i=0}^{n-1} F_{t_i} \diamond (W_{t_{i+1}}^H - W_{t_{i}}^H).
\end{equation*}

Taking a sequence of partitions $(\pi_n, n \in \mathbb{N})$, such that $|\pi_n|\xrightarrow[n \rightarrow \infty]{}0$ then $(S(F,\pi_n))_{n\in\mathbb{N}}$ is a Cauchy sequence in $L^2_\phi$. The limit of this sequence exists and is defined to be $\int_0^T F_s dW_s^H$. In other words, we define:

\begin{equation}\label{eq:2}
    \int_0^T F_s dW_s^H := \lim_{|\pi| \rightarrow 0} \sum_{i=0}^{n-1} F_{t_i} \diamond (W_{t_{i+1}}^H - W_{t_{i}}^H).
\end{equation}

 We define $\mathcal{L}_\phi(0,T)$ as the space where this process is well defined.

\begin{definition}
    Let $\mathcal{L}_\phi(0,T)$ be the family of stochastic processes such that:
    
    \begin{enumerate}
        \item $\mathbb{E}[||F||^2_\phi] < \infty$
        \item F is $\phi$-differentiable
        \item $\mathbb{E}[\int_0^T \int_0^T |D_s ^\phi F|^2 dsdt] < \infty$
        \small
        \item For each sequence of $(\pi_n, n \in \mathbb{N})$ with $|\pi_n| \xrightarrow[n \rightarrow \infty]{}0$
            \begin{equation*}
             \mathbb{E}\left[ \left|\sum_{i, j=0}^{n-1} \int_{t_i^{(n)}}^{t_{i+1}^{(n)}} \int_{t_j^{(n)}}^{t_{j+1}^{(n)}} D_s^\phi F_{t_i^{(n)}}^\pi D_t^\phi F_{t_j^{(n)}}^\pi d s d t-\int_0^T \int_0^T D_s^\phi F_t D_t^\phi F_s d s d t \right| \right] \xrightarrow[|\pi| \rightarrow 0]{}0
          \end{equation*}
        \normalsize 
        \item $\mathbb{E}[|F^\pi - F|^2_\phi] \xrightarrow[|\pi| \rightarrow 0]{}0$
    \end{enumerate}

\end{definition}

The following theorem sums up the results about the construction of the integration with respect to fBm.

\begin{thm}\label{thm:1}
    Let $(F_t, t \in [0,T])$ be a stochastic process such that $F_t \in \mathcal{L}_\phi(0,T)$. Then the limit in equation \ref{eq:2} exists and is defined to be exactly the integral of $F$ with respect to fBm, $\int_0^T F_s dW_s^H$.

    Furthermore, we have the following properties:
    \begin{equation*}
        \mathbb{E}\left[ \int_0^T F_s dW_s^H\right] = 0
    \end{equation*}
    and
    \begin{equation*}
        \left\lVert \int_0^T F_s dW_s^H \right\rVert^2_{\mathcal{L}_\phi(0,T)} := \mathbb{E}\left[ \left| \int_0^T F_s dW_s^H \right|^2 \right] = \mathbb{E}\left[\left(\int_0^T  D_t^\phi F_t dt \right)^2 + \lVert\mathds{1}_{[0,T]} F\rVert^2_\phi \right]
    \end{equation*}
\end{thm}

The idea behind the construction is to write the Riemann sums using the Wick product instead of the usual product. And we find a way to relate the Wick product and the usual one to get the stochastic integration in terms of the usual product. This is where the derivative comes into play: using Proposition $2.2$, we deduce the relationship and in Theorem $2.5$  the integration with the usual product.

\section{Itô's Formula for fBm when $H > \frac{1}{2}$}
\label{sec:ito_formula}

We now arrive to our most useful result, the Itô formula for the fractional Brownian Motion.

\begin{thm}\label{thm:2}
    Let $(F_t)_{t \in [0,T]}$ be a stochastic process in $\mathcal{L}_\phi(0,T)$ and $\eta=\int_0^T F_u dW_u^H$. Assume that $\alpha > 1-H$ and that there exists $C > 0$: 
    \begin{equation*}
        \mathbb{E}[|F_u-F_v|^2] \leq C |u-v|^{2\alpha},
    \end{equation*}
    where $|u-v| \leq \delta$ for some $\delta > 0$ and
    \begin{equation*}
        \lim_{0\leq u,v \leq t, |u-v| \rightarrow 0 } \mathbb{E}[|D_u^\phi(F_u- F_v)|^2] = 0
    \end{equation*}
    
    Let $f:\mathbb{R}^+ \times \mathbb{R} \rightarrow \mathbb{R}$, a $C^1$ function in the first variable and $C^2$ in the second variable. Furthermore, assume that these derivatives are bounded. It is also assumed that $\mathbb{E}[\int_0^T |F_s D_s^\phi \eta_s| ds] < \infty$ and that the process $(\frac{\partial f}{\partial x}(s, \eta_s)F_s, s\in [0,T])$, is in $\mathcal{L}_{\phi}[0,T]$. Then for $0 \leq t \leq T$, we have,
    
    \footnotesize
    \begin{equation*}
        f\left(t, \eta_t\right)=  f(0,0)+\int_0^t \frac{\partial f}{\partial s}\left(s, \eta_s\right) d s+\int_0^t \frac{\partial f}{\partial x}\left(s, \eta_s\right) F_s d W_s^{H} +\int_0^t \frac{\partial^2 f}{\partial x^2}\left(s, \eta_s\right) F_s D_s^\phi \eta_s d s \quad \text { a.s. }
    \end{equation*}
    \normalsize
\end{thm}

\begin{proof}
    See \cite{duncan-2000}, \cite{biagini-2008}
\end{proof}

If the process $F$ is actually a deterministic function, $a(s)$, then the theorem simplifies and we obtain:

\begin{corollary}
    Let $a \in L^2_\phi$ and define $\eta = \int_0^T a_s dW_s^H$. Take a function, $f:\mathbb{R^+} \times \mathbb{R} \rightarrow \mathbb{R}$, satisfying the conditions of Theorem \ref{thm:2} and let $\left(\frac{\partial f}{\partial x}(s, \eta_s)a_s, s \in [0,T]\right)$ be a process in $\mathcal{L}_\phi(0,T)$.
    
    Then,
    \footnotesize
    \begin{equation*}
        f(t,\eta_t) = f(0,0) + \int_0^t \frac{\partial f}{\partial s}(s, \eta_s) ds + \int_0^t \frac{\partial f}{\partial x}(s, \eta_s) a_s dW_s^H + \int_0^t \frac{\partial^2 f}{\partial x^2}(s, \eta_s) a_s \int_0^s \phi(s, v) a_v dvds \quad a.s.
    \end{equation*}
    \normalsize
\end{corollary}

With this corollary we can give a nice rule for the Itô formula when we only take the fBm. Choosing $a(s)\equiv 1$, the process $\eta_t$ becomes $\int_0^t dW_s^H = W_t^H$ and Itô's formula may be presented as follows,

\small
\begin{equation*}
    f(W_t^H) = f(W_0^H) + \int_0^t \frac{\partial f}{\partial x}(W_s^H) dW_s^H + H \int_0^t s^{2H-1} \frac{\partial^2 f}{\partial x^2}(W_s^H) ds \quad a.s.  
\end{equation*}
\normalsize

We give the Itô formula for a larger class of processes. 

\begin{thm}  
Let $\left(F_u, u \in[0, T]\right)$ satisfy the conditions of Theorem 3.1 and let $\mathbb{E} \sup _{0 \leq s \leq T}\left|G_s\right|<$ $\infty$. Denote $\eta_t=\xi+\int_0^t G_u d u+\int_0^t F_u d W_u^H, \xi \in \mathbb{R}$ for $t \in[0, T]$. Let $\left(\frac{\partial f}{\partial x}\left(s, \eta_s\right) F_s, s \in[0, T]\right) \in$ $\mathcal{L}(0, T)$. Then for $t \in[0, T]$
$$
\begin{aligned}
f\left(t, \eta_t\right)= & f(0, \xi)+\int_0^t \frac{\partial f}{\partial s}\left(s, \eta_s\right) d s+\int_0^t \frac{\partial f}{\partial x}\left(s, \eta_s\right) G_s d s \\
& +\int_0^t \frac{\partial f}{\partial x}\left(s, \eta_s\right) F_s d B_s^H+\int_0^t \frac{\partial^2 f}{\partial x^2}\left(s, \eta_s\right) F_s D_s^\phi \eta_s d s \quad \text { a.s }
\end{aligned}
$$
\end{thm}

\begin{proof}
    The proof is the same as for theorem 3.1. See \cite{duncan-2000}.
\end{proof}

\section{Itô-Wentzell type formula for $H > \frac{1}{2}$}
\label{sec:ito-wentzell}

In this section we present and prove an Itô-Wentzell formula for the fBm. We will use the integral constructed above for $H > \frac{1}{2}$. There are other ways to construct the stochastic integral with respect to fBm (see e.g. \cite{biagini-2008}, Ch.4), but the Itô formula obtained with this construction works for a larger class of processes. 

First we have to define the spaces:

\begin{definition}
    Let $\mathbb{D}^{1,p}$ be the closure of $\mathcal{S}$ with respect to the norm:
    \begin{equation*}
        ||F||_{1,p} = ||F||_p + || ||D^\phi F||_{L^2}||_p
    \end{equation*}

    Define $\mathbb{L}^{1,p} = L^p(\mathbb{R}^+;\mathbb{D}^{1,p}) $.

\end{definition}

Before presenting the proof of the Itô-Wentzell theorem, we  derive a stochastic rule for the product of two processes driven by fBm.

\begin{lemma}\label{lemma:2}
    Given two processes,
    \begin{equation*}
        X_t = X_0 + \int_0^t A_s^1 ds + \int_0^t B_s^1 dW_s^H
    \end{equation*}
    \begin{equation*}
        Y_t = Y_0 + \int_0^t A_s^2 ds + \int_0^t B_s^2 dW_s^H
    \end{equation*}
    we have, for each $t \in [0,T]$,
    \begin{equation*}
        d(X_t Y_t) = X_t dY_t + Y_t dX_t + (B_t^1 D_s^\phi Y_t + B_t^2 D_s^\phi X_t)dt.
    \end{equation*}
\end{lemma}

\begin{proof}
    In order to prove this lemma we define a new 2-dimensional process, $Z_t = (X_t, Y_t)$ and use the multi-dimensional Itô formula (see \cite{duncan-2000}, Theorem 4.6) to obtain the differential of $f(Z_t)$, where $f(x,y)=xy$.

\begin{flalign*}
    %\begin{equation*}
         \begin{split}
         df(Z_t) = \frac{\partial f}{\partial x}(X_t, Y_t) B_t^1 dW_t^H +  \frac{\partial f}{\partial y}(X_t, Y_t) B_t^2 dW_t^H
    %\end{equation*}
    %\begin{equation*}
        &{}+ \frac{\partial f}{\partial x}(X_t, Y_t) A_t^1 dt +  \frac{\partial f}{\partial y}(X_t, Y_t) A_t^2 dt \\
    %\end{equation*}
    %\begin{equation*}
        {} + B^2_t D_s^\phi X_t dt + B_t^1 D_s^\phi Y_t dt
        \end{split} &\\[10pt]
    %\end{equation*}
    %\begin{equation*}
        & = Y_t B_t^1 dW_t^H + X_t B_t^2 dW_t^H + Y_t A_t^1 dt + X_t A_t^2 dW_t^H
    %\end{equation*}
    %\begin{equation*}
        + B^2_t D_s^\phi X_t dt + B_t^1 D_s^\phi Y_t dt &\\[10pt]
    %\end{equation*}
    %\begin{equation*}
        & = X_t dY_t + Y_t dX_t + (B_t^1 D_s^\phi Y_t + B_t^2 D_s^\phi X_t)dt
    %\end{equation*}
\end{flalign*}

\end{proof}

\newpage

Now we present the main result of this paper. The notation $f'$ and $f''$ represent derivatives taken with respect to the space variable, $x$.

\begin{thm}[Itô-Wentzell formula for the fBm]
Consider two processes
    $X_t = X_0 + \int_0^t A_s ds + \int_0^t B_s dW^H_s$ and $F_t(x) = F_0(x) + \int_0^t G_s(x) ds + \int_0^t H_s(x) dW^H_s$
Assume the following conditions on the processes:

\begin{multicols}{2}

\begin{enumerate}
     
\item \begin{equation*}
    X \in \mathbb{L}^{1,4}
\end{equation*}

\item\begin{equation*}
    A \in L^4(\Omega)
\end{equation*}

\item\begin{equation*}
    B \in L^8(\Omega)
\end{equation*}

\item\begin{equation*}
    F \in \mathbb{L}^{1,4}(L^2(\mathbb{R}))
\end{equation*}

\item\begin{equation*}
    F \in C^2(\mathbb{R})
\end{equation*}

\item\begin{equation*}
    G \in L^{2}([0,T];L^2(\mathbb{R}))
\end{equation*}

\item\begin{equation*}
    H \in \mathbb{L}^{1,4}([0,T];L^2(\mathbb{R}))
\end{equation*}

\item\begin{equation*}
    \int_0^T \mathbb{E}\left[\sup_{x \in \mathbb{R}} |F_s'(x)|^4\right] ds< \infty
\end{equation*}

\columnbreak

\item\begin{equation*}
   \int_0^T \mathbb{E}\left[ \sup_{x \in \mathbb{R}} |F_s''(x)|^4 \right] ds< \infty
\end{equation*}

\item\begin{equation*}
    \int_0^T \mathbb{E}\left[ \left( \sup_{x \in \mathbb{R}} |D^\phi_s F_s'(x)|^2 \right) \right] ds < \infty
\end{equation*}

\item\begin{equation*}
    \int_0^T \mathbb{E}\left[ \left( \sup_{x \in \mathbb{R}} |(D^\phi_s F_s)'(x)|^4 \right) \right] ds < \infty
\end{equation*}

\item\begin{equation*}
    \int_0^T \mathbb{E}\left[ \sup_{x \in \mathbb{R}} |G_s(x)|^2 \right]ds < \infty
\end{equation*}

\item\begin{equation*}
    \int_0^T \mathbb{E}\left[\sup_{x \in \mathbb{R}} |H_s(x)|^4 \right] ds < \infty
\end{equation*}

\item\begin{equation*}
    \int_0^T \mathbb{E}\left[ \sup_{x \in \mathbb{R}} |H'_s(x)|^4\right] ds < \infty
\end{equation*}

\item\begin{equation*}
    \int_0^T \mathbb{E}\left[ \left( \sup_{x \in \mathbb{R}} |D^\phi_s H_s(x)|^4\right) \right] ds < \infty
\end{equation*}

\item\begin{equation*}
    \int_0^T \mathbb{E}\left[ \left( \sup_{x \in \mathbb{R}} |D^\phi_s X_s(x)|^4\right) \right] ds < \infty
\end{equation*}

\end{enumerate}

\end{multicols}

Then, for $t \in [0,T]$, we obtain the formula

\small
\begin{equation*}
    F_t(X_t) = F_0(X_0) + \int_0^t F_s'(X_s) A_s ds + \int_0^t F_s'(X_s) B_s dW^H_s + 
    \int_0^t F_s''(X_s) (D^\phi X)_s B_s ds
\end{equation*}
\begin{equation*}
     + \int_0^tG_s(X_s) ds + \int_0^t H_s(X_s) dW_s^H + 
     \int_0^t (D^\phi F)_s'(X_s) B_s ds + 
     \int_0^t H_s'(X_s) (D^\phi X)_s ds.
\end{equation*}
\normalsize
    
\end{thm}

\begin{proof}

We proceed using  ideas from Ocone and Pardoux \cite{ocone-1989}, section I.3. 

Define $\varphi \in C_c^\infty (\mathbb{R},\mathbb{R})$, such that $\int_{\mathbb{R}} \varphi(x)dx=1$ and construct a family of functions such that, $\varphi_\varepsilon (x) = \frac{1}{\varepsilon}\varphi\left(\frac{x}{\varepsilon}\right)$.

Using the usual Itô formula given in Theorem 3.2 we get, 

\small
\begin{equation*}
    \varphi_\varepsilon(X_t-x) = \varphi_\varepsilon(X_0-x) + \int_0^t \varphi'_\varepsilon(X_s-x) A_s ds
\end{equation*}
\begin{equation*}
    + \int_0^t \varphi'_\varepsilon(X_s-x) B_s dW_s^H + \int_0^t \varphi''_\varepsilon(X_s-x)B_s D_s^\phi X_s ds
\end{equation*}
\normalsize

Now, we use Lemma \ref{lemma:2} to obtain

\small

\begin{flalign*}
%\begin{equation*}
    & \int_\mathbb{R} F_t(x) \varphi_\varepsilon(X_t-x) dx &\\
%\end{equation*}
%\begin{equation*}
    &= \int_{\mathbb{R}} F_0(x) \varphi_\varepsilon(X_0-x)dx + \int_0^t \int_{\mathbb{R}} F_s(x) \varphi_\varepsilon'(X_s-x)A_s dxds &\\
%\end{equation*}
%\begin{equation*}
    &+  \int_0^t \int_{\mathbb{R}} F_s(x) \varphi_\varepsilon'(X_s-x)B_s dxdW_s^H + \int_0^t \int_{\mathbb{R}} F_s(x) \varphi_\varepsilon''(X_s-x)B_s D_s^\phi X_s dxds &\\
%\end{equation*}
%\begin{equation*}
    &+  \int_0^t \int_{\mathbb{R}} G_s(x) \varphi_\varepsilon(X_s-x) dxds + \int_0^t \int_{\mathbb{R}} H_s(x) \varphi_\varepsilon(X_s-x) dxdW_s^H &\\
%\end{equation*}
%\begin{equation*}
    &+  \int_0^t \int_{\mathbb{R}}  \varphi_\varepsilon'(X_s-x)B_s D_s^\phi F_s(x)dxds + \int_0^t \int_{\mathbb{R}} H_s(x) \varphi_\varepsilon'(X_s-x) D_s^\phi X_s dxds &\\
%\end{equation*}
\end{flalign*}
\normalsize

We need to use integration by parts and recall that $\varphi_\varepsilon(x)$ are compact support functions, so the boundary terms are 0.

We arrive at:
\small

\begin{flalign*}
%\begin{equation*}
    & \int_{\mathbb{R}} F_0(x) \varphi_\varepsilon(X_0-x)dx + \int_0^t \int_{\mathbb{R}} F_s'(x) \varphi_\varepsilon(X_s-x)A_sdxds &\\
%\end{equation*}
%\begin{equation*}
    & +  \int_0^t \int_{\mathbb{R}} F_s'(x) \varphi_\varepsilon(X_s-x)B_s dxdW_s^H +  \int_0^t \int_{\mathbb{R}} F_s''(x) \varphi_\varepsilon(X_s-x)B_s D_s^\phi X_s dxds &\\
%\end{equation*}
%\begin{equation*}
    & +  \int_0^t \int_{\mathbb{R}} G_s(x) \varphi_\varepsilon(X_s-x) dxds + \int_0^t \int_{\mathbb{R}} H_s(x) \varphi_\varepsilon(X_s-x) dxdW_s^H &\\
%\end{equation*}
%\begin{equation*}
    & +   \int_0^t \int_{\mathbb{R}}  \varphi_\varepsilon(X_s-x)B_s (D_s^\phi F_s)'(x)dxds + \int_0^{t} \int_{\mathbb{R}} \varphi_\varepsilon(X_s-x) H_s'(x) D_s^{\phi} X_s dxds 
%\end{equation*}
\end{flalign*}

\normalsize

Notice that in the integrals over $\mathbb{R}$ we have a convolution between $\varphi_\varepsilon(x)$ and some function. We also know that $\varphi_\varepsilon(x)$ are identity approximations; therefore: $f * \varphi_\varepsilon(x) \longrightarrow f$, when $\varepsilon \rightarrow 0$.

Let's start with the deterministic integrals, namely

\small
\begin{equation*}
    \int_0^t \int_{\mathbb{R}} F_s'(x)\varphi_\varepsilon(X_s-x)A_s dx ds
\end{equation*}
\normalsize

We'll prove that this term converges in $L^2(\Omega)$.

From the continuity of $F'$ in $x$ for fixed $(\omega, s)$, we have that $\int_{\mathbb{R}} F_s'(x)\varphi_\varepsilon(X_s-x)A_s dx$ tends to $F'(X_s)A_s$ when $\varepsilon$ tends to zero. Moreover,

\begin{equation*}
    \left|\int_{\mathbb{R}} F_s'(x)\varphi_\varepsilon(X_s-x)A_s dx \right| \leq \sup_{x\in \mathbb{R}} |F_s'(x)A_s|
\end{equation*}

Then, 

\small
\begin{flalign*}
%\begin{equation*}
   & \left( \mathbb{E} \left[ \left| \int_0^t \left(\int_{\mathbb{R}} F_s'(x) \varphi_\varepsilon(X_s-x) A_s dx - F_s'(X_s)A_s \right) ds \right|^2 \right] \right)^\frac{1}{2} &\\
%\end{equation*}
%\begin{equation*}
    & \leq \int_0^t \left( \mathbb{E} \left[ \left|\int_{\mathbb{R}} \varphi_\varepsilon(X_s-x) F_s'(x)A_s dx - F_s'(X_s)A_s \right|^2 \right] \right)^\frac{1}{2} ds &\\
%\end{equation*}
%\begin{equation*}
    & = \int_0^t \left( \mathbb{E} \left[ \left|\int_{\mathbb{R}} \varphi_\varepsilon(X_s-x) F_s'(x) dx - F_s'(X_s) \right|^2 |A_s|^2 \right] \right)^\frac{1}{2} ds &\\
%\end{equation*}
%\begin{equation*}
    & \leq \int_0^t \left(\mathbb{E} \left[ \left|\int_{\mathbb{R}} \varphi_\varepsilon(X_s-x) F_s'(x) dx - F_s'(X_s) \right|^4 \right]\right)^{\frac{1}{4}} \left(\mathbb{E} \left[ |A_s|^4 \right]\right)^{\frac{1}{4}}ds 
%\end{equation*}    
\end{flalign*}

\normalsize

For these inequalities we used Hölder and Minkowski's inequalities.

This quantity converges to $0$ when $\varepsilon \rightarrow 0$ because $(F_s'* \varphi_\varepsilon)(X_s) \rightarrow F_s'(X_s)$. We have to apply Lebesgue dominated convergence theorem (DCT) to take the integral of the limit. For this we use conditions (2) and (8).

We proceed in a similar way to prove the convergence of the other integrals in $ds$.

To prove that

\small
\begin{equation*}
    \int_0^t \int_{\mathbb{R}} F_s''(x) \varphi_\varepsilon(X_s-x)B_s D_s^\phi X_s dxds \longrightarrow \int_0^t \int_{\mathbb{R}} F_s''(X_s) B_s D_s^\phi X_s dxds
\end{equation*}
\normalsize

we proceed as above,

\small
\begin{flalign*}
%\begin{equation*}
    & \left( \mathbb{E}\left[ \left| \int_0^t \left( \int_{\mathbb{R}} F_s''(x) \varphi_\varepsilon(X_s-x) B_sD_s^\phi X_s dx - F_s''(X_s)B_sD_s^\phi X_s \right) ds \right|^2 \right] \right)^\frac{1}{2} &\\
%\end{equation*}
%\begin{equation*}
    & \leq \int_0^t \left( \mathbb{E} \left[ \left|\int_{\mathbb{R}} \varphi_\varepsilon(X_s-x) F_s''(x)B_sD_s^\phi X_s dx - F_s''(X_s)B_sD_s^\phi X_s \right|^2 \right] \right)^\frac{1}{2} ds &\\
%\end{equation*}
%\begin{equation*}
    & = \int_0^t \left( \mathbb{E} \left[ \left|\int_{\mathbb{R}} \varphi_\varepsilon(X_s-x) F_s''(x) dx - F_s''(X_s) \right|^2 |B_s|^2|D_s^\phi X_s|^2 \right] \right)^\frac{1}{2} ds &\\
%\end{equation*}
%\begin{equation*}
    & \leq \int_0^t \left(\mathbb{E} \left[ \left|\int_{\mathbb{R}} \varphi_\varepsilon(X_s-x) F_s''(x) dx - F_s''(X_s) \right|^4 \right] \right)^{\frac{1}{4}} \left(\mathbb{E} \left[ |B_s|^4|D_s^\phi X_s|^4 \right]\right)^{\frac{1}{4}}ds &\\
%\end{equation*}
%\begin{equation*}
    & \leq \int_0^t \left(\mathbb{E} \left[ \left|\int_{\mathbb{R}} \varphi_\varepsilon(X_s-x) F_s''(x) dx - F_s''(X_s) \right|^4 \right]\right)^{\frac{1}{4}} \left(\mathbb{E} \left[ |B_s|^8\right]\right)^{\frac{1}{8}} \left(\mathbb{E}\left[|D_s^\phi X_s|^8 \right]\right)^{\frac{1}{8}}ds 
%\end{equation*}
\end{flalign*}
\normalsize

In this case use the hypotheses (1), (3) and (9).

The fourth term given by

\small
\begin{equation*}
     \int_0^t \int_{\mathbb{R}} G_s(x) \varphi_\varepsilon(X_s-x) dxds
\end{equation*}
\normalsize

converges to 

\small
\begin{equation*}
    \int_0^t G_s(X_s) ds
\end{equation*}

So

\footnotesize
\begin{equation*}
    \left( \mathbb{E}\left[ \left| \int_0^t \left( \int_{\mathbb{R}} G_s(x) \varphi_\varepsilon(X_s-x) - G_s(X_s) \right)ds \right|^2 \right] \right )^\frac{1}{2}
%\end{equation*}
%\begin{equation*}
    \leq \int_0^t \left( \mathbb{E} \left[ \left|\int_{\mathbb{R}} \varphi_\varepsilon(X_s-x) G_s(x)dx - G_s(X_s)\right|^2 \right] \right)^\frac{1}{2} ds
\end{equation*}

\normalsize

In this case we use hypothesis (6) and (12).

Next,

\small
\begin{equation*}
\int_0^t \int_{\mathbb{R}}  \varphi_\varepsilon(X_s-x)B_s (D_s^\phi F_s)'(x)dxds \longrightarrow \int_0^t \int_{\mathbb{R}}  B_s (D_s^\phi F_s)'(X_s)dxds 
\end{equation*}
\normalsize

The convergence follows from the inequalities,

\small
\begin{flalign*}
%\begin{equation*}
    & \left( \mathbb{E}\left[ \left| \int_0^t \left(\int_{\mathbb{R}} (D_s^\phi F_s)'(x)\varphi_\varepsilon(X_s-x) B_sdx - (D_s^\phi F_s)'(X_s)B_s \right) ds \right|^2 \right] \right)^\frac{1}{2} &\\
%\end{equation*}
%\begin{equation*}
    & \leq \int_0^t \left( \mathbb{E} \left[ \left|\int_{\mathbb{R}} \varphi_\varepsilon(X_s-x) (D_s^\phi F_s)'(x)B_s dx -(D_s^\phi F_s)'(X_s)B_s \right|^2 \right] \right)^\frac{1}{2} ds &\\
%\end{equation*}
%\begin{equation*}
    & = \int_0^t \left( \mathbb{E} \left[ \left|\int_{\mathbb{R}} \varphi_\varepsilon(X_s-x) (D_s^\phi F_s)'(x) dx -(D_s^\phi F_s)'(X_s) \right|^2 |B_s|^2 \right] \right)^\frac{1}{2} ds &\\
%\end{equation*}
%\begin{equation*}
    & \leq \int_0^t \left(\mathbb{E} \left[ \left|\int_{\mathbb{R}} \varphi_\varepsilon(X_s-x) (D_s^\phi F_s)'(x) dx -(D_s^\phi F_s)'(X_s) \right|^4 \right]\right)^{\frac{1}{4}} \left(\mathbb{E} \left[ |B_s|^4\right]\right)^{\frac{1}{4}}ds 
%\end{equation*}
\end{flalign*}
\normalsize

This last one follows from (3) and (11).

Finally,

\small

$$\int_0^{t} \int_{\mathbb{R}} \varphi_\varepsilon(X_s-x) H_s'(x) D_s^{\phi} X_s dxds$$
converges to 
 $$\int_0^{t} H_s'(X_s) D_s^{\phi} X_s ds $$
since we have

\begin{flalign*}
& \left( \mathbb{E}\left[ \left| \int_0^{t} \left( \int_{\mathbb{R}} \varphi_\varepsilon(X_s-x) H_s'(x) D_s^{\phi} X_s dx - H_s'(X_s) D_s^{\phi} X_s \right) ds \right|^2 \right] \right)^\frac{1}{2} &\\
& \leq \int_0^{t} \mathbb{E}\left[ \left|  \int_{\mathbb{R}} \varphi_\varepsilon(X_s-x) H_s'(x) D_s^{\phi} X_s dx - H_s'(X_s) D_s^{\phi} X_s 
 \right|^2 \right] ds &\\
&= \int_0^{t} \left(\mathbb{E}\left[ \left|  \int_{\mathbb{R}} \varphi_\varepsilon(X_s-x) H_s'(x) D_s^{\phi} X_s dx - H_s'(X_s) \right|^2 \left| D_s^{\phi} X_s 
 \right|^2 \right] \right)^\frac{1}{2} ds &\\
& \leq \int_0^{t} \left(\mathbb{E}\left[ \left|  \int_{\mathbb{R}} \varphi_\varepsilon(X_s-x) H_s'(x) dx - H_s'(X_s) \right|^4 \right]\right)^\frac{1}{4} \left( \mathbb{E}\left[\left| D_s^{\phi} X_s 
 \right|^4 \right] \right)^\frac{1}{4} ds 
\end{flalign*}

\normalsize

This follows from (1) and (14).

The convergence for the Riemann integral terms is done. Now we turn our attention to the stochastic integral terms.

For the integrals in $dW^H_s$ we will use the isometry given by Theorem 2.5.

%\begin{equation*}
%    \mathbb{E}\left[ \left | \int_0^T F_s dW_s^H \right|^2 \right] =       \mathbb{E}\left[\int_0^T \int_0^T D_s^\phi F_t D_t^\phi F_s dsdt + ||\mathds{1}_{[0,T]} F||^2_\phi \right]
%\end{equation*}

For example, we want to prove that the process $$\int_0^t \int_{\mathbb{R}} H_s(x) \varphi_{\varepsilon}(X_s-x) dx d W_s^H$$ converges in the $L^2$ norm to the process $$\int_0^t H_s(X_s) dW_s^H$$.

Taking the $L^2$ norm we get,

\footnotesize
\begin{equation*}
     \left(\mathbb{E}\left[ \left | \int_0^t \left( \int_{\mathbb{R}} H_s(x) \varphi_{\varepsilon}(X_s-x)d x - H_s(X_s) \right) dW_s^H \right|^2 \right]\right)^{\frac{1}{2}}
\end{equation*}

%-------------------------
%\begin{equation*}
%     \mathbb{E}\left[\int_0^t \int_0^t D_u^\phi \left( \int_{\mathbb{R}} H_v(x) \varphi_{\varepsilon}(X_v-x) d x  - H(X_v) \right) D_v^\phi \left(  \int_{\mathbb{R}} H_u(x) \varphi_{\varepsilon}(X_u-x)d x -  H(X_u)\right) dudv \right] + 
%\end{equation*}
%-------------------------

%VERSÃO EM DUAS LINHAS

%\begin{equation*}
%     \begin{aligned}
%     \mathbb{E}\bigg[ \bigg( \int_0^t  D_s^\phi \left(  \int_{\mathbb{R}} H_s(x) \varphi_{\varepsilon}(X_s-x) & d x -  H_s(X_s)\right)  \bigg)^2 ds \bigg] + 
%     \\ &\mathbb{E} \left[ \left|\left| \left( \int_{\mathbb{R}} H_s(x) \varphi_{\varepsilon}(X_s-x)  d x  - H(X_s)  \right) \right|\right|^2_\phi  \right]
%     \end{aligned}
%\end{equation*}

%VERSÃO NUMA SÓ LINHA

\begin{equation*}
     = \left(\mathbb{E}\bigg[ \bigg( \int_0^t  D_s^\phi \left(  \int_{\mathbb{R}} H_s(x) \varphi_{\varepsilon}(X_s-x) d x -  H_s(X_s)\right) ds \bigg)^2  \bigg] + 
     \mathbb{E} \left[ \left|\left|
     \int_{\mathbb{R}} H_s(x) \varphi_{\varepsilon}(X_s-x)  d x  - H_s(X_s)   \right|\right|^2_\phi  \right]\right)^\frac{1}{2}
\end{equation*}

\begin{equation*}
     \leq \left(\mathbb{E}\bigg[ \bigg( \int_0^t  D_s^\phi \left(  \int_{\mathbb{R}} H_s(x) \varphi_{\varepsilon}(X_s-x) d x -  H_s(X_s)\right) ds \bigg)^2  \bigg]\right)^\frac{1}{2} + 
     \left(\mathbb{E} \left[ \left|\left|
     \int_{\mathbb{R}} H_s(x) \varphi_{\varepsilon}(X_s-x)  d x  - H_s(X_s)   \right|\right|^2_\phi  \right]\right)^\frac{1}{2}
\end{equation*}

%-------------------------------------------
%\begin{equation*}
%    \mathbb{E} \left[ \left|\left| \left( \int_{\mathbb{R}} H_s(x) \varphi_{\varepsilon}(X_s-x)  d x  - H(X_s)  \right) \right|\right|^2_\phi  \right] = 
%\end{equation*}

%\begin{equation*}
%     \mathbb{E}\left[ \left (  \int_0^t D_s^\phi \left( \int_{\mathbb{R}} H_u(x) \varphi_{\varepsilon}(X_u-x) d x - H_u(X_u) \right) duds\right)^2 \right] + \mathbb{E} \left[ \left|\left| \left( \int_{\mathbb{R}} H_s(x) \varphi_{\varepsilon}(X_s-x) d x -  H(X_s) \right) \right|\right|^2_\phi  \right]
%\end{equation*}

%\begin{equation*}
%    \mathbb{E} \left[ \left|\left| \left( \int_{\mathbb{R}} H_s(x) \varphi_{\varepsilon}(X_s-x) d x -  H(X_s) \right) \right|\right|^2_\phi  \right] 
%\end{equation*}
%---------------------------------------------

\normalsize

We focus on the first term:

\footnotesize
\begin{equation*}
    \left(\mathbb{E}\left[ \left (\int_0^t D_s^\phi \left(  \int_{\mathbb{R}} H_s(x) \varphi_{\varepsilon}(X_s-x) d x -  H_s(X_s) \right) ds\right)^2 \right]\right)^\frac{1}{2}
\end{equation*}

\begin{equation*}
   \leq \int_0^t \left( \mathbb{E}\left[ \left (D_s^\phi \left(  \int_{\mathbb{R}} H_u(x) \varphi_{\varepsilon}(X_u-x) d x -  H_u(X_u) \right) \right)^2 \right] \right)^\frac{1}{2} ds
\end{equation*}

%And now the properties for the derivative $D_s^\phi$

%\begin{equation*}
%    \int_0^t \int_0^t \mathbb{E}\left[ \left (D_s^\phi \left(  \int_{\mathbb{R}} H_u(x) \varphi_{\varepsilon}(X_u-x) d x -  H(X_u) \right) \right)^2 \right] duds =
%\end{equation*}

\begin{equation*}
    \begin{aligned}
         = \int_0^t \bigg(\mathbb{E}\bigg[\bigg(  \int_{\mathbb{R}} \varphi_{\varepsilon}(X_s-x)D_s^\phi H_s(x) + & \varphi_{\varepsilon}'(X_s-x)H_s(x) + H_s(x)\varphi_{\varepsilon}(X_s-x)D_s^\phi X_s  d x - \\     
        & (D_s^\phi H_s)(X_s)D_s^\phi X_s - H_s'(X_s)D_s^\phi(X_s) - H_s(X_s)D_s^\phi X_s  \vphantom{\int_{\mathbb{R}}} \bigg)^2\bigg] \bigg)^\frac{1}{2}ds
    \end{aligned}
\end{equation*}

\begin{equation*}
    \begin{aligned}
        = \int_0^t  \bigg(\mathbb{E}\bigg[\bigg(  \int_{\mathbb{R}} \varphi_{\varepsilon}(X_s-x)D_s^\phi H_s(x) + & \varphi_{\varepsilon}'(X_s-x)H_s(x)D_s^\phi X_s + H_s(x)\varphi_{\varepsilon}(X_s-x)D_s^\phi X_s  d x - \\     
        & (D_s^\phi H_s)(X_s)D_s^\phi X_s - H_s'(X_s)D_s^\phi(X_u) - H_s(X_s)D_s^\phi X_s  \vphantom{\int_{\mathbb{R}}} \bigg)^2\bigg] \bigg)^\frac{1}{2}ds
    \end{aligned}
\end{equation*}

\normalsize

The last two equations are obtained by using the rules of differentiation for $D_s^\phi$ that can be found in (\cite{duncan-2000}, (3.5)-(3.8)).

In order not to have an overflow of notation, we'll set 

\small
\begin{equation*}
   A = \int_{\mathbb{R}} \varphi_{\varepsilon}(X_s-x)D_s^\phi H_s(x) dx - D^\phi_sH_s(X_s),
\end{equation*}

\begin{equation*}
    B = \int_{\mathbb{R}} \varphi_{\varepsilon}(X_s-x)H_s'(x)D_s^\phi X_s dx - H_s'(X_s)D_s^\phi X_s,
\end{equation*}

\begin{equation*}
    C= \int_{\mathbb{R}} H_s(x)\varphi_{\varepsilon}(X_s-x)D_s^\phi X_s dx - H_s(X_s) D_s^\phi X_s. 
\end{equation*}

Then, we get, 

\begin{equation*}
         \int_0^t  \left(\mathbb{E}\left[\left(A+B+C  \right)^2\right] \right)^\frac{1}{2} ds \leq \int_0^t \left(3 \mathbb{E}\left[A^2+B^2+C^2\right] \right)^\frac{1}{2}ds \leq 3^\frac{1}{2}\int_0^t (\mathbb{E}[A^2])^\frac{1}{2}+(\mathbb{E}[B^2])^\frac{1}{2}+(\mathbb{E}[C^2])^\frac{1}{2}ds
\end{equation*}

The convergence of each of these terms follows from the DCT and hypothesis (1), (7), (13), (14), (15) and (16). 

%Aqui em cima falta me o expoente. Vou ter que usar a desigualdade para sqrt e depois as hipotses como fiz para a convergencia do termo G_s. 

%%  FAZER USANDO A IGUALDADE NOS INTEGRAIS, É ASSIM QUE CONSTA NA TESE, O REVISOR PEDIU QUE USASSE A DESIGUALDADE (A+B+C)^2 < 3(A^2+B^2+C^2) %%  

%\begin{equation*}
%         \int_0^t \mathbb{E}\left[A^2+B^2+C^2 +2AB + 2AC + 2BC\right] ds 
%\end{equation*}

\normalsize

%We recall the linearity of both the integral on $ds$ and the expectation. We, now, have to prove the convergence of all six terms.

%The convergence of the first three terms comes from the Lebesgue dominated convergence theorem and conditions (1), (13), (14), (15) and (16).

%In the last 3, we first apply Hölder's inequality and get,

%\small

%\begin{equation*}
%    \mathbb{E}\left[|AB|\right] \leq (\mathbb{E}[|A|^2])^{\frac{1}{2}}(\mathbb{E} [|B|^2])^{\frac{1}{2}}
%\end{equation*}

%\normalsize

%Again by the Lebesgue dominated convergence we obtain the convergence of all these terms when $\varepsilon \rightarrow 0$

To finish up we still have to derive the convergence in the term:

\small

\begin{equation*}
    \mathbb{E} \left[ \left|\left| \left( \int_{\mathbb{R}} H_s(x) \varphi_{\varepsilon}(X_s-x) d x -  H(X_s) \right) \right|\right|^2_\phi  \right]. 
\end{equation*}

\normalsize

Again, for the sake of clarity we'll set:

\small

\begin{equation*}
    A_s = \int_\mathbb{R} H_s(x) \varphi_\varepsilon(X_s-x)dx - H_s(X_s).
\end{equation*}

\normalsize

Then,

\begin{flalign*}
%\begin{equation*}
    & \mathbb{E} \left[ \left|\left| \left( \int_{\mathbb{R}} H_s(x) \varphi_{\varepsilon}(X_s-x) d x -  H(X_s) \right) \right|\right|^2_\phi  \right] &\\
%\end{equation*}
%\begin{equation*}
    & = \mathbb{E} \left[\int_0^t \int_0^t A_u A_v \phi(u,v) dudv\right]&\\
%\end{equation*}
%\begin{equation*}
    & \leq \mathbb{E} \left[\int_0^t \int_0^t A_u A_v |u-v|^{2H-2} dudv\right] &\\
%\end{equation*}
%\begin{equation*}
    & \leq \int_0^t \int_0^t  \mathbb{E} \left[ A_u A_v \right] |u-v|^{2H-2}dudv &\\
%\end{equation*}
%\begin{equation*}
    & \leq \int_0^t \int_0^t  (\mathbb{E} \left[ A_u^2\right])^{\frac{1}{2}} (\mathbb{E}\left[A_v^2\right])^{\frac{1}{2}} |u-v|^{2H-2}dudv 
%\end{equation*}    
\end{flalign*}

\normalsize

The convergence follow from condition (13) and the DCT.

The last term we need to analyze is 

\small
\begin{equation*}
    \int_0^t \int_{\mathbb{R}} F_s'(x) \varphi_\varepsilon(X_s-x)B_s dxdW_s^H
\end{equation*}
which converges to 
\begin{equation*}
    \int_0^t F_s'(X_s) B_s dW_s^H
\end{equation*}
\normalsize

by a similar argument.
    
\end{proof}

\section{A class of Stochastic Differential Equations}
\label{sec:class_sde}

We prove existence and uniqueness of solution for this equation following an idea in Ocone and Pardoux's paper \cite{ocone-1989}.

\begin{equation}\label{eq:61}
      X_t=X_0 + \int_0^t b(s,X_s) ds + \sigma W_t^H ,
\end{equation}
where $b:[0,T] \times \Omega \times \mathbb{R} \rightarrow \mathbb{R}$ and $\sigma \in \mathbb{R}$.

We define $\varphi_t(x)$ as the flow of the equation (\ref{eq:61}), which in this case reduces to:

\begin{equation*}
    \varphi_t(x) = x + \sigma W_t^H
\end{equation*}

Using the notation introduced by Ocone and Pardoux \cite{ocone-1989}, we define the operator:

\begin{equation*}
    (\varphi_t^{*-1} b)(t, \omega, x) = \left(\frac{\partial \varphi_t}{\partial x}\right)^{-1}(x) b(t, \omega, \varphi_t(x))
\end{equation*}

In our case this simply becomes 

\begin{equation*}
    (\varphi_t^{*-1} b)(t, \omega, x) = b(t, \omega, x + \sigma W_t^H(\omega)) = b_\omega (t, x + \sigma W_t^H)
\end{equation*}
because $\frac{\partial \varphi_t}{\partial x}=1$.

It can be shown that

\begin{equation}\label{eq:62}
    dY_t= (\varphi_t^{*-1} b)(t, Y_t) = b(t, x + \sigma W_t^H)dt
\end{equation}
has an unique solution. Furthermore, using Itô-Wentzell's formula, and under suitable conditions on $b$, $\varphi(Y_t)$ is a solution of equation (\ref{eq:61}).

\begin{prop}\label{prop:61}
    Equation (19) has a unique, non-exploding solution $\{Y_t, t\geq0\}$, provided that $b$ is Lipchitz in $x$ and $|b(t,\omega, x)| \leq C(1+|x|)$. 
\end{prop}

\begin{proof}
    This follows from the fact that $\varphi_t^{*-1}b$  is Lipschitz in $x$ and by Picard-Lindelöf Theorem.
%    The solution is also non-exploding because 
%    \begin{equation*}
%        |b(t,\omega, x)| \leq C(1+|x|)
%    \end{equation*}
%    and the fact that 
%    \begin{equation*}
%        \sup_{t\leq T} |\varphi_t(x)| \leq  |x| + |\sigma| C_{p,\gamma}R^{\frac{1}{2p}}.
%    \end{equation*}
%    So we get that,
%    \begin{equation*}
%        |\varphi_t^{*-1}b(t,\omega, x)| = |b(t, x + \sigma W_t^H)| \leq C(1 + |x + \sigma W_t^H|) \leq
%    \end{equation*}
%    \begin{equation*}
%        C(1 + |x| + |\sigma|C_{p,\gamma}R^{\frac{1}{2p}})
%    \end{equation*}
\end{proof}

This result shows that $Y_t$ exists and is unique. Moreover, $Y_t=\varphi_t^{-1}(X_t)$, where $\varphi_t^{-1} = x - \sigma W_t^H$. This flow is invertible and differentiable so, by proving that $Y_t$ exists we proved that $X_t = \varphi_t(Y_t)$ also exists and is unique.

\begin{thm}
    Equation has a unique solution, $X_t$, provided that the conditions of Proposition \ref{prop:61} are satisfied.
\end{thm}

There are some results in the literature about  conditions under which SDEs driven by fBm have unique solutions. Our last remark will focus in comparing the condition we derived in this chapter and the ones found in two papers \cite{zhang-2020} and \cite{jien-2009}. 

In \cite{zhang-2020} Zhang and Yuan studied a similar family of equations, $dX_t = b(t,X_t)dt + \sigma dW^H_t$, with $H>\frac{1}{2}$. In their existence and uniqueness theorem the authors have three assumptions. The first one requiring boundness on the first derivative of $b$; the second one asks for $b(t, x)\geq h_1 x^{-\alpha}$, $\alpha > \frac{1}{H}-1$ and $h_1 > 0$; finally the third assumption is: $b(t, x) \leq h_2(x + 1)$, $h_2 > 0$. We can see that the assumptions in our theorem are weaker that the ones presented in this paper.

On the other hand, in \cite{jien-2009}, the authors consider a larger class of SDEs, $dX_t = b(t,X_t)dt + \sigma(t,X_t) dW^H_t$, with $H\in(0,1)$ and have much weaker assumptions: it is only assumed that $b(t,x)$ is Lipchitz in the second variable. Our theorem is not as general, and asks not only $b(t,x)$ to be Lipchitz, but also $b(t,x) < (1 + |x|)$. Their approach consists in using a family of transformations and an anticipating Girsanov theorem for the fBM to obtain an ordinary differential equation (ODE) and prove existence and uniqueness for this ODE. Our proofs are different.

\vskip 5mm

\bf Funding \rm This research work was supported by the FCT project \\UIDB/00208/2000.

\vskip 5mm

\bf Other Interests \rm The author declares no financial or non-financial competing interests.

%\bf Acknowledgments 

\bibliographystyle{unsrtnat} 
\bibliography{cas-refs}





% \end{thebibliography}
\end{document}